\documentclass{amsart}
\usepackage{amssymb,latexsym}
\usepackage{amsmath}
\usepackage{amscd}
\usepackage{graphicx}
 \usepackage{color}
\usepackage{enumerate}
\numberwithin{equation}{section}
\theoremstyle{plain}
 \newtheorem{theorem}{Theorem}[section]
 \newtheorem{gslemma}[theorem]{Swing lemma}
 \newtheorem{strswlemma}[theorem]{Theorem (Strong swing lemma)} 
 \newtheorem{sslemma}[theorem]{Slim swing lemma} 
 \newtheorem{proposition}[theorem]{Proposition}

\theoremstyle{definition}
 \newtheorem{definition}[theorem]{Definition}
 \newtheorem{remark}[theorem]{Remark}

\newenvironment{enumeratei}{\begin{enumerate}[\quad\upshape (i)]} {\end{enumerate}}

\newcommand \ustar [1] {{{#1}^\ast}}
\newcommand \dstar [1] {{{#1}_\ast}}
\newcommand \sseq {\textup{(\ref{defSsda}$\&$\ref{defSsdb})}-sequence}
\newcommand \pseq {SL-sequence}
\newcommand \spseq {SSL-sequence}
\newcommand \gseq {SLG-sequence}

\newcommand \sspan {\textup{(\ref{defSsda}$\&$\ref{defSsdb})}-span}

\newcommand \spspan {SSL-span}
\newcommand \ba {\boldsymbol\alpha}
\DeclareMathOperator{\Eye}{Eyes}
\newcommand \tbf[1] {\textbf{#1}}  
\newcommand\ideal[1]{\mathord\downarrow #1}
\newcommand \con {\textup{con}}
\newcommand \inp {{\mathfrak p}}
\newcommand \inq {{\mathfrak q}} 
\newcommand \ins {{\mathfrak s}} 
\newcommand \inr {{\mathfrak r}} 
\newcommand\set [1]{\{#1\}}
\newcommand \tuple [1] {\langle #1 \rangle}
\newcommand \pair [2] {\tuple{#1,#2}}

\newcommand\red[1]{{\textcolor{red}{#1}}}

\begin{document}
\title[Swing lattice game  and the swing lemma]
{Swing lattice game and a short proof of the swing lemma for planar semimodular lattices }

\author[G.\ Cz\'edli]{G\'abor Cz\'edli}
\email{czedli@math.u-szeged.hu}
\urladdr{http://www.math.u-szeged.hu/\textasciitilde{}czedli/}
\address{University of Szeged\\ Bolyai Institute\\Szeged,
Aradi v\'ertan\'uk tere 1\\ Hungary 6720}
\author[G.\ Makay]{G\'eza Makay}
\email{makayg@math.u-szeged.hu}
\urladdr{http://www.math.u-szeged.hu/\textasciitilde{}makay/}
\address{University of Szeged\\ Bolyai Institute\\Szeged,
Aradi v\'ertan\'uk tere 1\\ Hungary 6720}

\thanks{This research was supported by
NFSR of Hungary (OTKA), grant number K 115518}

\begin{abstract} The swing lemma, due to G.\ Gr\"atzer for slim semimodular lattices and extended by G.\ Cz\'edli and  G.\ Gr\"atzer for all planar semimodular lattices, describes the congruence generated by a prime interval in an efficient way. Here we present a new proof for this lemma, which is shorter than the earlier two. 
Also, motivated by the swing lemma and mechanical pinball games with flippers, we construct an online game called Swing lattice game. A computer program realizing this game is available from the authors' websites. 
\end{abstract}


\dedicatory{Dedicated to the eighty-fifth birthday of B\'ela Cs\'ak\'any}


\date{\red{\tbf{\Large{July 22, 2016}} {\Small{(submitted to Acta Sci.\ Math. (Szeged): July 15, 2016)}}}}

\maketitle

\section{Introduction}
The last decade has witnessed a rapid development of the theory of planar semimodular lattices; see the bibliographic section in the present paper and see many additional papers referenced in the book chapter Cz\'edli and Gr\"atzer~\cite{czgggltsta}. Also, see \cite{czgggltsta} for a survey and for all concepts not defined here. 
Since every planar semimodular lattice can be obtained from a slim semimodular lattice, a particularly intensive attention was paid to slim (hence necessarily planar) semimodular lattices; definitions will be given later. 
%

\begin{figure}[ht] 
\centerline
{\includegraphics[scale=1.0]{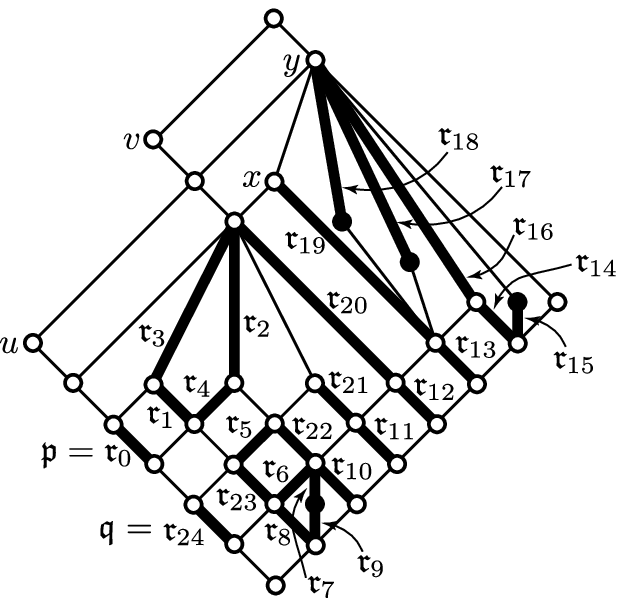}}
\caption{A \pseq{} from $\inp$ to $\inq$ in a planar semimodular lattice}\label{figpsa}
\end{figure}

\subsection*{First target: the swing lemma}
Semimodularity is \emph{upper semimodularity}, that is, a lattice is \emph{semimodular} if the implication  $x\preceq y\Rightarrow x\vee z\preceq y\vee z$ holds for all of its elements $x$, $y$ and $z$. A lattice $L$ is \emph{planar} if it has a planar Hasse-diagram.
Although Cz\'edli~\cite{czgdiagrectext}, which is a long paper, assigns a unique planar diagram to an arbitrary planar semimodular lattice,
we will not rely on \cite{czgdiagrectext} in the present elementary paper; we always assume that a planar diagram of our lattice is \emph{fixed} somehow. (Some concepts, like ``left'' or ''eye'', will depend on the choice of the diagram, but this fact will not cause any trouble.)  
\emph{Edges} $\inp=[a,b]$ of (the diagram of) $L$ are
also called \emph{prime intervals}. For a prime interval $\inp=[a,b]$ of $L$, we denote $a$ and $b$ by $0_\inp$ and $1_\inp$, respectively. It follows from semimodularity that the edges divide the area of the diagram into quadrangles, which we call \emph{$4$-cells}; more details will be given later. 
The least congruence collapsing (the two elements of) a prime interval $\inp$ is denoted by $\con(\inp)$ or $\con(0_\inp,1_\inp)$. In order to characterize whether $\con(\inp)$ collapses another prime interval $\inq$ or not, we need the following definition.

\begin{definition}\label{defSsd} Let $\inr$ and $\ins$ be distinct prime intervals of a planar  semimodular lattice such that they belong to the same $4$-cell $S$.
\begin{enumeratei}
\item\label{defSsda} If $\inr$ and $\ins$ are opposite sides of $S$ then $\inr$ is \emph{cell-perspective} to $\ins$.
\item\label{defSsdb} If $1_\inr=1_\ins$,  $1_\inr$ has at least three lover covers, and $0_\ins$ is neither the leftmost, nor the rightmost lower cover of $1_\inr$, then $\inr$ \emph{swings} to $\ins$.
\item\label{defSsdc} If $0_\inr=0_\ins$,  $0_\inr$ has at least three covers, and $1_\ins$ is neither the leftmost, nor the rightmost  cover of $0_\inr$, then $\inr$ \emph{tilts} to $\ins$.
\end{enumeratei}
For $n\in\set{0,1,2,\dots}$, a sequence 
\begin{equation}
\vec{\inr}: \inr_0,\inr_1,\dots,\inr_n
\label{eqsseqpseq}
\end{equation}
of prime intervals is called an 
\emph{\pseq}  if 
for each $i\in\set{1,\dots,n}$, $\inr_{i-1}$ is cell-perspective to or swings to  or tilts to $\inr_i$. (The acronym ``SL'' comes from ``swing lemma''.) In $\vec{\inr}$, $\inr_0$ and $\inr_n$ play a distinguished role, and we often say that $\vec{\inr}$ is an 
\emph{\pseq{} from $\inr_0$ to $\inr_n$}. It is \emph{cyclic} if $\inr_0=\inr_{n}$.
\end{definition}

\begin{figure}[ht]
\centerline
{\includegraphics[scale=1.0]{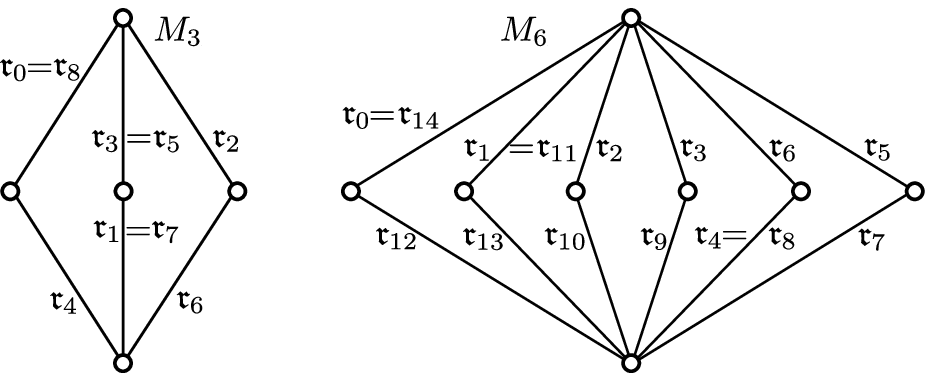}}
\caption{Cyclic \spseq{}s in $M_3$ and $M_6$}\label{figcyclM6}
\end{figure}

While \eqref{defSsda} describes a symmetric relation,  \eqref{defSsdb} and  \eqref{defSsdc} do not. To see some examples, consider the planar semimodular lattice in Figure~\ref{figpsa}. Then $\inr_{11}$ and  
$\inr_{12}$ are mutually cell-perspective to each other,  $\inr_{2}$ and  $\inr_{3}$  mutually swing to each other,  so do $\inr_{16}$ and $\inr_{17}$; $\inr_{8}$ tilts to $\inr_{9}$, and $\inr_{6}$ swings to $\inr_{7}$. However, $\inr_{9}$ does not tilt to $\inr_{8}$ and  $\inr_{7}$ does not swing to $\inr_{6}$. The sequence  $\inr_{0}$, $\inr_{1}$, \dots, $\inr_{24}$ is an \pseq{} from $\inp$ to $\inq$, and it remains an \pseq{} if we omit $\inr_{7}$ and~$\inr_{8}$. In Figure~\ref{figcyclM6}, the sequence $\inr_0$, $\inr_1$, \dots, $\inr_{14}=\inr_0$ is a cyclic \pseq{} in $M_6$.

\begin{remark}\label{remswtlslopes} If the diagram of $L$ belongs to the class $\mathcal C_1$ defined in Cz\'edli~\cite{czgdiagrectext}, then \eqref{defSsdb} and \eqref{defSsdc} from Definition~\ref{defSsd} can be formulated in the following, more visual way; see \cite{czgdiagrectext}. Namely, for \emph{distinct} edges $\inr$ and $\ins$ of the \emph{same} 4-cell, 
\begin{enumerate}
\item[(ii)$'$]  $\inr$ \emph{swings} to $\ins$ if  $1_\inr=1_\ins$ and the slope of $\ins$ is neither $45^\circ$, nor $135^\circ$.
\item[(iii)$'$]  $\inr$ \emph{tilts} to $\ins$ if  $0_\inr=0_\ins$ and the slope of $\ins$ is neither $45^\circ$, nor $135^\circ$.
\end{enumerate}
\end{remark}

Note that the diagrams in this paper belong to $\mathcal C_2$, which is a subclass of $\mathcal C_1$; the reader may want (but does not need) to see \cite{czgdiagrectext} for details. Note also that, by  \cite[Observation 6.2]{czgdiagrectext}, the condition that ``the slope of $\ins$ is neither $45^\circ$, nor $135^\circ$'' above is equivalent to the condition 
that ``the slope of $\ins$ is strictly between neither $45^\circ$ and $135^\circ$''. 
The following result was proved in Cz\'edli and Gr\"atzer~\cite{czgggswing}. 

\begin{gslemma}[Cz\'edli and Gr\"atzer~\cite{czgggswing}]\label{genswinglemma} Let $L$ be a planar semimodular lattice, and let $\inp$ and $\inq$ be prime intervals of $L$. Then $\pair{0_\inq}{1_\inq}\in\con(\inp)$ if and only if there is an \pseq{} from $\inp$ to $\inq$.
\end{gslemma}

For a bit stronger but more technical variant of the swing lemma, see Theorem~\ref{thmstrongswinglemma}. Although the proof in Gr\"atzer and Cz\'edli~\cite{czgggswing} is short, it relies on a particular case, which we will call \emph{slim swing lemma}; see Section~\ref{sectionslimsl}. The slim swing lemma is due to  Gr\"atzer~\cite{swinglemma} and there is another proof in Cz\'edli~\cite{czgdiagrectext}, but both these papers give long and complicated proofs. Furthermore, the proof in \cite{czgggswing} uses a lemma from Cz\'edli~\cite{czgrepres} that needed a three-page long proof.
So, if \cite{swinglemma} (or the relevant part of \cite{czgdiagrectext}) and the three pages from \cite{czgrepres} are also counted, the proof of the swing lemma is quite long. Our main goal is to give a much shorter proof.

\subsection*{Second target: the Swing lattice game} 
Section~\ref{sectiongame} describes our online game called \emph{Swing lattice game}. Its purpose is to increase the popularity of lattice theory in an entertaining way. Besides the swing lemma, the game  is also motivated by mechanical pinball games with flippers. A computer program  realizing the game is available from the authors websites. Note that the game has a screen saver mode. 
Another motivation for the Swing lattice game is that this paper is devoted to Professor Emeritus B\'ela Cs\'ak\'any, who is not only a highly appreciated algebraist and the scientific father or grandfather of almost all algebraists in Szeged, but he is interested in mathematical games. This interest is witnessed by, say,   Cs\'ak\'any~\cite{csbhun} and
Cs\'ak\'any and Juh\'asz~\cite{csbjr}.

\color{black}

\section{Preliminaries and a survey}
Besides collecting some known facts that will be needed in our proof, the majority of this  section gives a restricted survey on planar semimodular lattices. For a more extensive survey, the reader can resort to Cz\'edli and Gr\"atzer~\cite{czgggltsta}.

A lattice $L$ is \emph{slim} if $J(L)$, the poset of join-irreducible elements of $L$, contains no 3-element antichain.  By convention, both slim lattices and planar lattices are \emph{finite} by definition. By a \emph{diamond} we mean an $M_3$ (sub)lattice; see on the left of Figure~\ref{figcyclM6}. We know from Cz\'edli and Gr\"atzer~\cite[Lemma 3-4.1]{czgggltsta} that slimness implies planarity. Hence, we will drop ``planar'' from ``slim planar semimodular''. 
A sublattice $S$ of a lattice $L$ is a \emph{cover-preserving sublattice} if for any $a,b\in S$, $a\prec_S b$ implies that $a\prec_L b$. 
By Cz\'edli and Gr\"atzer~\cite[Thm.\ 3-4.3]{czgggltsta} or, originally, by Cz\'edli and Schmidt~\cite{czgschtJH} and Gr\"atzer and Knapp~\cite{gratzerknapp1}, 
a planar semimodular lattice is slim iff it contains no diamond  iff it contains no cover-preserving diamond. 
For example, by Cz\'edli and Gr\"atzer~\cite[Theorem 3-4.3]{czgggltsta} or by Proposition~\ref{propczsdB},  Figure~\ref{figslima} is a slim semimodular lattice. Also, if we omit the four black-filled elements from
the planar semimodular lattice given Figure~\ref{figpsa}, then we obtain a slim semimodular lattice.

\begin{figure}[ht]
\centerline
{\includegraphics[scale=1.0]{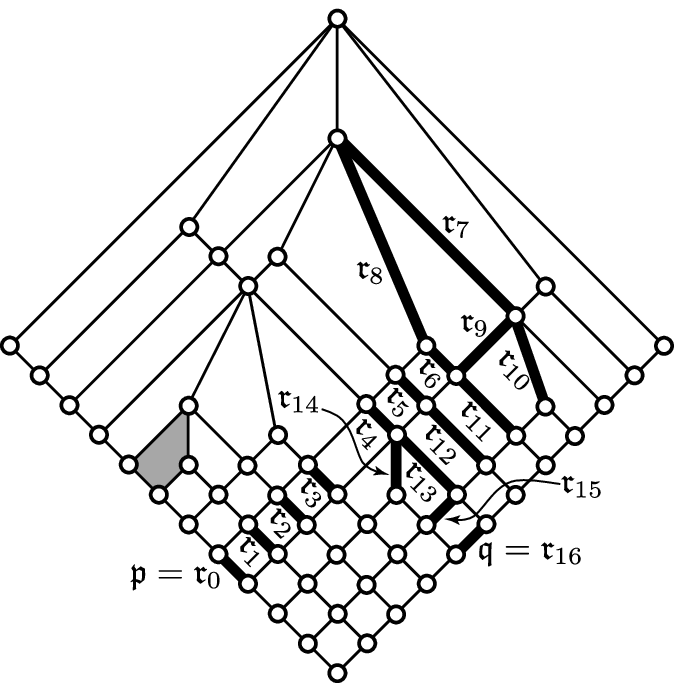}}
\caption{An \sseq{} from $\inp$ to $\inq$ in a slim semimodular (actually, a slim rectangular) lattice}\label{figslima}
\end{figure}

In (the fixed planar diagram of)  a planar semimodular lattice $L$, let $a<b$ but $a\nprec b$. If $C_1$ and $C_2$ are maximal chains in the interval $[a,b]$ such that $C_1\cap C_2=\set{a,b}$ and  every element of $C_2\setminus\set{a,b}$ is on the right of $C_1$, then the elements of $[a,b]$ that are simultaneously on the right of $C_1$ and on the left of $C_2$ form a \emph{region} of (the diagram of) $L$. 
Note that $C_1\cup C_2$ is a subset of this region. For example, the elements belonging to the   grey area in the second lattice  of  Figure~\ref{figindstep} form a region denoted by $R$. 
We know from Kelly and Rival~\cite[Prop.\ 1.4 and Lemma 1.5]{kellyrival} that, in (the fixed planar diagram of) a planar lattice,
\begin{equation}
\parbox{6.4cm}{every interval is a region and every region is a cover-preserving sublattice.}
\label{eqtxtKellyRival}
\end{equation}
If we drop the condition $C_1\cap C_2=\set{a,b}$ above, then we obtain a union (actually, a so-called glued sum) of regions, which is clearly still a sublattice. More precisely, for elements $a<b$ in a planar lattice $L$, 
\begin{equation}
\parbox{7.6cm}{if $C_1$ and $C_2$ are maximal chains in $[a,b]$ such that every element of $C_2$ is on the right of $C_1$, then $\{x\in [a,b]: x$ is on the right of $C_1$ and on the left of $C_2\}$ is a cover-preserving sublattice of $L$.}
\label{eqtxtmltKR}
\end{equation}
For more about planar lattice diagrams (of planar semimodular lattices), the reader may but need not look into Kelly and Rival~\cite{kellyrival} (or Cz\'edli and Gr\"atzer \cite{czgggltsta}). %
Minimal regions are called \emph{cells}.  For example, the grey area in Figure~\ref{figslima} and that in the first lattice of Figure~\ref{figindstep} are cells; actually, they are  \emph{$4$-cells} since they are formed by four vertices and four edges. In (the planar diagram of) a planar semimodular lattice, every cell is a 4-cell; see Gr\"atzer and Knapp~\cite[Lemma 4]{gratzerknapp1}. Hence, by Cz\'edli and Schmidt~\cite[Lemma 13]{czgschvisual}, 
\begin{equation}
\parbox{9 cm}{If $x$ and $y$ are neighboring lower covers of an element $z$ in a planar semimodular lattice, then $\set{x\wedge y, x, y, z}$ is a 4-cell.}
\label{eqtxttwcodhBtn}
\end{equation}
A 4-cell can be turned into a diamond by adding a new element into its interior. The new element is called an \emph{eye} and we refer to this step as \emph{adding an eye}. Note that after adding an eye, one ``old'' 4-cell is replaced with two new 4-cells. 
We know from Cz\'edli and Gr\"atzer~\cite[Cor.\ 3-4.10]{czgggltsta} that 
\begin{equation}
\parbox{8.5cm}{
every planar semimodular $L$ lattice is obtained from a slim semimodular lattice $L_0$ by adding eyes, one by one.}
\label{eqtxtbddyS}
\end{equation}
Note that $L_0$ is a sublattice of $L$. Although $L_0$ is not unique as a sublattice, it is unique up to isomorphism; see \cite[Lemma 3-4.8]{czgggltsta}. 
We call $L_0$ the \emph{full slimming} of $L$, while $L$ is an \emph{antislimming} of $L_0$. 
Note that  the full slimming of $L$  can be obtained from $L$ by omitting all eyes. 
For example, the full slimming $L_0$ of the planar semimodular lattice $L$ given in Figure~\ref{figpsa} is obtained by omitting the four black-filled elements. Conversely, we obtain $L$ from $L_0$ by adding eyes, four times. 
Based on, say, Gr\"atzer and Knapp~\cite[Lemma 8]{gratzerknapp1}, eyes are easy to recognize: an element $x$ of a planar semimodular lattice is an eye if and only if $x$ is doubly (that is, both meet and join) irreducible, its unique lower cover, denoted by $\dstar x$, has at least three covers, and $x$ is neither the leftmost, nor the rightmost cover of $\dstar x$.  
\begin{figure}[ht]
\centerline
{\includegraphics[scale=1.0]{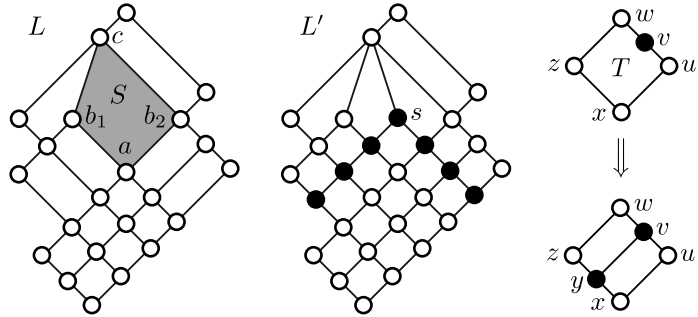}}
\caption{Inserting a fork}\label{figaddfork}
\end{figure}

\begin{definition}\label{defstrswtlt} 
Let  $\inr$ and $\ins$ distinct edges of the \emph{same} 4-cell in (the planar diagram of) a planar semimodular lattice $L$, and let $\Eye(L)$ denote the set of eyes of $L$. 
\begin{enumerate}
\item[(ii)$'$]  $\inr$ \emph{strongly swings} to $\ins$ if  $\inr$ swings to $\ins$ and, in addition, the implication $0_\inr\in \Eye(L) \Longrightarrow 0_\ins\in\Eye(L)$ holds.
\end{enumerate}
The sequence $\vec\inr$ in \eqref{eqsseqpseq} will be called an \spseq{} if for each $i\in\set{1,\dots,n}$, $\inr_{i-1}$ is cell-perspective to or tilts to or  strongly swings to $\inr_i$. (The acronym ``SSL'' comes from ``strong swing lemma''.)
\end{definition}

In a planar semimodular lattice, 
\begin{equation}
\text{every \spseq{} is a \pseq{},}
\label{eqtxtspseqpseq}
\end{equation}
but not conversely. For example, in Figure~\ref{figpsa}, the two-element sequence $\inr_{18}$, $[x,y]$  is an \pseq{} but not an \spseq. Now, we are in the position to formulate the following theorem. By \eqref{eqtxtspseqpseq}, it implies Lemma~\ref{genswinglemma}, the swing lemma.

\begin{strswlemma}[Cz\'edli and Gr\"atzer~\cite{czgggswing}]\label{thmstrongswinglemma} 
If $L$ is a planar semimodular lattice and  $\inp$ and $\inq$ are prime intervals of $L$, then the following two implications hold.
\begin{enumeratei}
  \item\label{thmstrongswinglemmaa} If there exists an \pseq{} from $\inp$ to $\inq$ $($in particular, if there is an \spseq{} from $\inp$ to $\inq)$, then $\pair{0_\inq}{1_\inq}\in\con(\inp)$.
  \item\label{thmstrongswinglemmab} Conversely, if $\pair{0_\inq}{1_\inq}\in\con(\inp)$, then  there exists an \spseq{} from $\inp$ to $\inq$.
\end{enumeratei}
\end{strswlemma}

By \eqref{eqtxtbddyS},  in order to have a satisfactory insight into planar semimodular lattices, it suffices to describe the slim ones. In order to do so, we need  the following  concepts. 

Based on Cz\'edli and Schmidt~\cite{czgschvisual}, Figure~\ref{figaddfork} visualizes how we \emph{insert a fork} into a 4-cell $S$ of a slim semimodular lattice $L$ in order to obtain a new slim semimodular lattice $L'$. First, we add a new element $s$ into the interior of $S$. Next, we add two lower covers of $s$ that will be on the lower boundary of $S$ as indicated in the figure. 
Finally, we do a series of steps: as long as there is a chain $u\prec v\prec w$ such that $T=\set{x=z\wedge u, z, u, w=z\vee u}$ is a 
4-cell in the original $L$ and $x\prec z$ at the present stage, then we insert a new element $y$ such that $x\prec y\prec z$ and $y\prec v$; see on the right of the figure. The new elements of $L'$, that is, the elements of $L'\setminus L$, are the black-filled ones in Figure~\ref{figaddfork}. 

A doubly irreducible element $x$ on the boundary of a slim semimodular lattice is called a \emph{corner} if it has a unique upper cover $\ustar x$ and a unique lower cover $\dstar x$, $\ustar x$ covers exactly two elements, and $\dstar x$ is covered by exactly two elements.  For example, after omitting the black-filled elements from Figure~\ref{figpsa}, there are exactly two corners, $u$ and $v$. Note that there is no corner in the slim semimodular lattice given by Figure~\ref{figslima}. A \emph{grid} is the (usual diagram of the) direct product of two finite non-singleton chains.


\begin{proposition}[Cz\'edli and Schmidt~\cite{czgschvisual}]\label{propczsdB} 
Every slim semimodular lattice with at least three elements can be obtained from a grid such that 
\begin{enumeratei}
\item\label{propczsdBa}  first we add finitely many forks one by one, 
\item\label{propczsdBb}  and then we remove corners, one by one, finitely many times.
\end{enumeratei} 
Furthermore, all lattices obtained in this way are slim and semimodular.
\end{proposition}

Note that by Cz\'edli and Schmidt~\cite[Prop.\ 2.3]{czgschslim2}, the lattices we obtain by \eqref{propczsdBa} but without  \eqref{propczsdBb} are exactly the \emph{slim rectangular lattices} introduced by  Gr\"atzer and Knapp~\cite{gratzerknapp3}; see Figure~\ref{figslima} for an example. We can add eyes to these lattices; what we obtain in this way are the so-called \emph{rectangular lattices}; see ~\cite[Prop.\ 2.3]{czgschslim2} and Gr\"atzer and Knapp~\cite{gratzerknapp3}.

\section{Slim swing lemma}\label{sectionslimsl}
The slim lemma was first stated and proved only for slim semimodular lattices; to make a terminological distinction, we will refer to it as the ``slim swing lemma".
\begin{definition}
The sequence $\inr$ from \eqref{eqsseqpseq} is an \emph{\sseq} if for each $i\in\set{1,\dots,n}$, $\inr_{i-1}$ is cell-perspective to or swings to $\inr_i$.
\end{definition}

For example, the edges $\inr_0$, $\inr_1$, \dots, $\inr_{16}$ in Figure~\ref{figslima} form an \sseq.
In a planar semimodular lattice, every \sseq{} is an \pseq{} but, in general, not conversely.  
Since every element of a slim semimodular lattice has at most two covers by  
 Gr\"atzer and Knapp~\cite[Lemma 8]{gratzerknapp1}, tilts are impossible in \emph{slim} semimodular lattices. That is,
\begin{equation}
\parbox{8.2cm}{In a \emph{slim} semimodular lattice, \pseq{}s, \spseq{}s, and  \sseq{}s are exactly the same.}
\label{eqtxthgmnB}
\end{equation}
Therefore, the following statement is a particular case of Lemma~\ref{genswinglemma}.

\begin{sslemma}[Gr\"atzer~\cite{swinglemma}]\label{specswinglemma} Let $L$ be a slim semimodular lattice, and let $\inp$ and $\inq$ be prime intervals of $L$. Then $\pair{0_\inq}{1_\inq}\in\con(\inp)$ if and only if there is an \sseq{} from $\inp$ to $\inq$.
\end{sslemma}

Note that  Gr\"atzer~\cite{swinglemma} states this lemma in another way. In order to see that our version implies his version trivially, two easy observations will be given below. 
For prime intervals $\inp$ and $\inq$, if $1_\inp\vee 0_\inq= 1_\inq$ and  $1_\inp\wedge 0_\inq=0_\inp$, then $\inp$ is \emph{up-perspective} to $\inq$ and $\inq$ is \emph{down-perspective} to $\inp$. \emph{Perspectivity} is the disjunction of up-perspectivity and down-perspectivity. As an important property of \sseq s, we claim that, for prime intervals $\inp$ and $\inq$ in a finite semimodular lattice $L$,
\begin{equation}
\parbox{9cm}{If $\inp$ is up-perspective to $\inq$,  then there is an \sseq{} $\vec\inr=\tuple{\inr_0,\dots,\inr_n}$ from $\inp$ to $\inq$ such that $\inr_{i-1}$ is upward cell-perspective to $\inr_i$ for all $i\in\set{1,\dots, n}$. Conversely, if there is such an $\vec\inr$, then $\inp$ is up-perspective to $\inq$.}
\label{eqtxbsLsP}
\end{equation}
The second part of \eqref{eqtxbsLsP} is trivial. In order to see its first part, assume that  $\inp$ is up-perspective to $\inq$, and  pick  maximal chain $0_\inp=x_0\prec x_1\prec\dots\prec x_n=0_\inq$. For $i\in\set{1,\dots, n}$,  $\set{x_{i-1}, x_i, 1_\inp\vee x_{i-1},1_\inp\vee {x_i}}$ is a covering square by semimodularity. (For more details, if necessary, see the explanation around Figure 1 in Cz\'edli and Schmidt~\cite{czgschthowtoderive}.)  Covering squares are 4-cells by Cz\'edli and Gr\"atzer~\cite[Thm.\ 3-4.3(v)]{czgggltsta},  whence 
there is an \sseq{} $\vec\inr$ from $\inp$ to $\inq$ with the required property. This proves 
\eqref{eqtxbsLsP}. 

It is clear from Cz\'edli and Schmidt~\cite[Lemma 2.8]{czgschtJH}, and it can also be derived from Proposition~\ref{propczsdB} by induction,  that in a slim semimodular lattice,
\begin{equation}
\parbox{8.5cm}{For a \emph{repetition-free} \sseq{} $\vec\inr$ from \eqref{eqsseqpseq}  in a \emph{slim} semimodular lattice, if $\inr_{i-1}$ is up-perspective to $\inr_i$, then $\inr_{j-1}$ is up-perspective to $\inr_j$ for all $j\in\set{1,2,\dots, i}$.}
\label{eqtxtlttEcS}
\end{equation}

Now it is clear that, by  \eqref{eqtxbsLsP} and \eqref{eqtxtlttEcS}, Lemma~\ref{specswinglemma} and 
its original version in Gr\"atzer~\cite{swinglemma} mutually imply each other.

\begin{figure}[ht]
\centerline
{\includegraphics[scale=1.0]{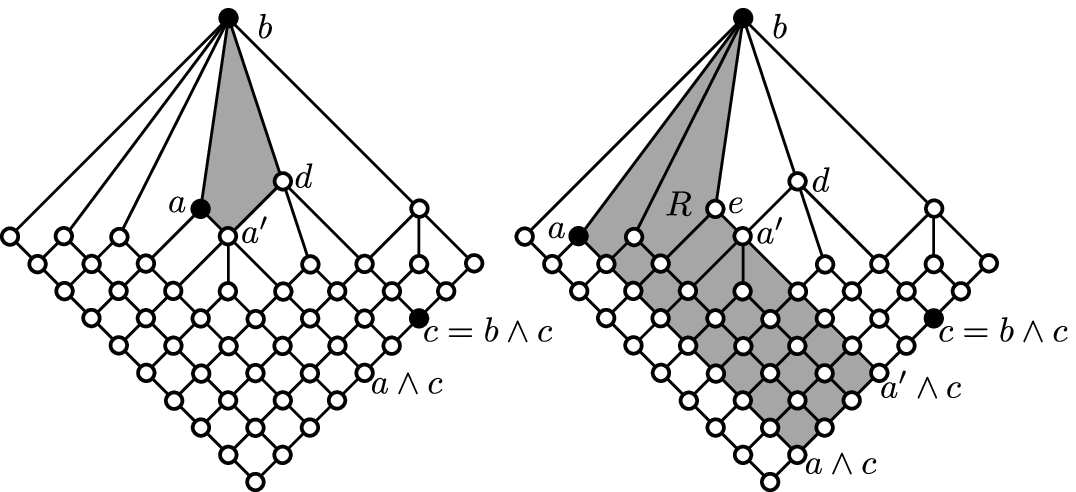}}
\caption{Illustration for \eqref{eqtxtindstp}}\label{figindstep}
\end{figure}

\section{The short proof}
\begin{proof}[Proof of Theorem~\ref{thmstrongswinglemma}]
Part \eqref{thmstrongswinglemmaa}
 follows easily from known results and \eqref{eqtxtspseqpseq}. For example, it follows from Cz\'edli~\cite[Theorems 3.7 and  5.5 (or 7.3)]{czgtrajcolor} and Cz\'edli~\cite[Thm.\ 2.2, Cor.\ 2.3, and Prop.\ 5.10]{czgdiagrectext}, 
; however, the reader will certainly find it more convenient to observe that both $\con(w_\ell,t)$ and $\con(w_r,t)$  
collapses the pairs $\pair{s_i}{t}$ of $\sf S^{(n)}_7$ in \cite[Fig.\ 1]{czgtrajcolor} by routine calculations.

Before proving part \eqref{thmstrongswinglemmab}, some preparation is needed.
For $n\in\set{3,4,5,\dots}$, the  $n+2$-element modular lattice of length 2 is denoted by $M_n$. For example, $M_3$ and $M_6$ are given in Figure~\ref{figcyclM6}. As this figure suggests, it is easy to see that, for $n\in\set{3,4,5,\dots}$,
\begin{equation}
\text{$M_n$ has a cyclic \spseq{} that contains all edges.}
\label{eqtxtM6}
\end{equation}
For a prime interval $\inr$ and elements  $u\leq v$ of a planar semimodular lattice $L$, we will say that $\inr\,$ \emph{\spspan s} (respectively, $\inr\,$ \emph{\sspan s}) the interval $[u,v]$ if there is an $n\in\set{0,1,2,\dots}$ and  there exists a maximal chain $u=w_0\prec w_1\prec \dots \prec w_n=v$ in $[u,v]$ such that, for each $i\in\set{1,\dots,n}$, there is an \spseq{} (respectively, an \sseq{}) from $\inr$ to $[w_{i-1},w_i]$. First, we focus on \sspan{}ning. We claim the following;  see Figure~\ref{figindstep}.
\begin{equation}
\parbox{8.5cm}{If $a,b,c$ are elements of a \emph{slim} semimodular lattice $K$ such that $a\prec b$, then $[a,b]$ \sspan s $[a\wedge c,  b\wedge c]$.}
\label{eqtxtindstp}
\end{equation}
We prove \eqref{eqtxtindstp} by induction on $|K|$. The base of the induction, $|K|\leq 4$, is obvious. We can assume that $c\leq b$, because otherwise we can replace $c$ with $b\wedge c$. Actually, we assume that $c<b$ but $c\nleq a$, since otherwise the satisfaction of \eqref{eqtxtindstp} is trivial. Pick an element $d$ such that $c\leq d\prec b$; see Figure~\ref{figindstep}. Since $c\nleq a$ and $a\prec b$,  $a$ and $d$ are distinct lower covers of $b$. By left-right symmetry, we assume that $a$ is to the left of $d$. There are two cases to consider.

First, assume that among the lower covers of $b$, $a$ is immediately to the left of $d$;  see the first lattice of Figure~\ref{figindstep}. Let $a'=a\wedge d$. By \eqref{eqtxttwcodhBtn}, $\set{a',a,d,b}$ is a 4-cell. Hence, there is a ``one-step'' \sseq{} from $[a,b]$ to $[a',d]$, which consists of  a downwards cell-perspectivity.
Observe that  $a\wedge c=a\wedge (d\wedge c)=(a\wedge d)\wedge c=a'\wedge c$ and the principal ideal $\ideal d$ does not contain $a$. Hence, $|\ideal d|<|K|$. Thus, the induction hypotheses yields that $[a',d]$ \sspan s   $[a'\wedge c, c]=[a\wedge c, b\wedge c]$. 
This is witnessed by some \sseq s; combining them with the one-step \sseq{} mentioned above, we conclude that $[a,b]$ \sspan s $[a\wedge c,  b\wedge c]$, as required.

Second, assume that there is a lower cover of $b$ strictly to the right of $a$ and to the left of $d$. Let $e$ denote the rightmost one of these lower covers and let $a':=e\wedge d$; see the second lattice in Figure~\ref{figindstep}. Since $\set{a',e,d,b}$ is a 4-cell by \eqref{eqtxttwcodhBtn},  there is a one-step \sseq{} from $[e,b]$ to $[a',d]$. Combining it with a sequence of swings from $[a,b]$ to $[e,b]$, we obtain a \sseq{} from $[a,b]$ to $[a',d]$.
Applying the induction hypothesis to $\ideal d$, we obtain that $[a',d]$ \sspan s 
$[a'\wedge c, d\wedge c]= [a'\wedge c, b\wedge c]$. Taking the above-mentioned \sseq{} into account, it follows that $[a,b]$ \sspan s 
$[a'\wedge c, c]= [a'\wedge c, b\wedge c]$.
We know from  Cz\'edli and Gr\"atzer~\cite[Exercise 3.4]{czgggltsta} and it also follows from \eqref{eqtxtKellyRival} that $a\wedge d\leq e\wedge d=a'$. Hence, 
$a\wedge c= a\wedge d\wedge c \leq a'\wedge c$.
In the interval $[a\wedge c,b]$, let $C_2$ be a maximal chain such that  $\set{a'\wedge c, a', e}\subseteq C_2$.

The elements of $[a\wedge c,b]$ on the left of $C_2$ form a cover-preserving sublattice $L_1$, because \eqref{eqtxtmltKR} applies for the leftmost maximal chain of $[a\wedge c,b]$ and $C_2$. Since $a$ is on the left of $e$, $a\in L_1$ by Kelly and Rival~\cite[Prop.\ 1.6]{kellyrival}.
Pick a maximal chain $C_1$ in $L_1$ such that $a\in C_1$, and let $R$ denote the cover-preserving sublattice of $L_1$ determined by $C_1$ and $C_2$ in the sense of \eqref{eqtxtmltKR}. Since $d$ is strictly on the right of $e\in C_2$, $d\notin R$ by Kelly and Rival~\cite[Prop.\ 1.6]{kellyrival}. Thus, $|R|<|K|$. Hence, the induction hypothesis applies for $\tuple{R, a, b, a'\wedge c}$ in the role of $\tuple{K, a, b, c}$, and we obtain that 
$[a,b]$ \sspan s $[a\wedge c,a'\wedge c]$ in $R$. Since $R$ is a cover-preserving sublattice and also a region, the same holds in $K$. Therefore, since $[a,b]$ \sspan s both  $[a\wedge c,a'\wedge c]$ and $[a'\wedge c,b\wedge c]$, it \sspan s $[a\wedge c,b\wedge c]$. This proves \eqref{eqtxtindstp}.

Next, we claim that 
\begin{equation}
\parbox{8.2cm}{If $a,b,c$ are elements of a \emph{planar} semimodular lattice $L$ such that $a\prec b$, then $[a,b]$ \spspan s $[a\wedge c,  b\wedge c]$.}
\label{eqtxtindPln}
\end{equation}
By \eqref{eqtxthgmnB},  \eqref{eqtxtindPln} generalizes \eqref{eqtxtindstp}. 
In order to prove  \eqref{eqtxtindPln}, let $K$ denote  the full slimming of $L$. Its elements and edges will be called \emph{old}, while the rest of elements and edges are \emph{new}; this terminology is explained by \eqref{eqtxtbddyS} and the paragraph following it. The new elements are exactly the eyes.
As in the proof of \eqref{eqtxtindstp}, we can assume that $c<b$ but $c\nleq a$.
First, we deal only with the case where $[a,b]$ is an \emph{old edge}.
Since (the segments of) \sseq s are also \spseq s by \eqref{eqtxthgmnB},  \eqref{eqtxtM6} implies that  
\begin{equation}
\parbox{9.7cm}{if $\ins_1$ and $\ins_2$ are old edges and there is an \sseq{} from $\ins_1$ to $\ins_2$ in $K$, then there is an \spseq{} from $\ins_1$ to $\ins_2$ in $L$.}
\label{eqtxtKthnLvn}
\end{equation}
Hence, for an old prime interval $\ins$ and old elements $u\leq v$, 
\begin{equation}
\parbox{8.8cm}{if $\ins$ \sspan s $[u,v]$ in $K$, then $\ins\,$  \spspan s $[u,v]$ in $L$.}
\label{eqtxtLbnsPns}
\end{equation}
If $c$ is also an old element, then $\set{a\wedge c,b\wedge c}\subseteq K$, so  the validity of \eqref{eqtxtindPln} follows from \eqref{eqtxtindstp} and \eqref{eqtxtLbnsPns}. Hence, we can assume that $c$ is an eye. Let  $\ustar c$ and $\dstar c$  stand for its (unique) cover and lower cover, respectively; they are old elements. Since $c<b$ and $c$ is meet-irreducible, $\ustar c \leq b$.  \eqref{eqtxtindstp} yields that $[a,b]$  \sspan s  $[a\wedge \dstar c, b\wedge \dstar c]=[a\wedge \dstar c,\dstar c]$ in $K$. Since $c\nleq a$, $a\wedge c<c$. Using that  $c$ is join-irreducible, we have that $a\wedge c=a\wedge \dstar c$. Hence, by \eqref{eqtxtLbnsPns}, 
\begin{equation}
\text{$[a,b]$  \spspan s $[a\wedge c,\dstar c]=[a\wedge \dstar c,\dstar c]$ in $L$.}
\label{eqtxtabspstcc}
\end{equation}
On the other hand, $a\wedge \ustar c< \ustar c$, since otherwise $c<\ustar c\leq a$ would contradict $c\nleq a$.  \eqref{eqtxtindstp} yields that $[a,b]$  \sspan s  $[a\wedge \ustar c, b\wedge \ustar c]=[a\wedge \ustar c,\ustar c]$.  Thus, we can pick an old element $w$ such that $a\wedge \ustar c\leq w\prec \ustar c$ and there is an \sseq{} from $[a,b]$ to $[w,\ustar c]$ in $K$. By \eqref{eqtxtKthnLvn}, we have an \spseq{} from $[a,b]$ to $[w,\ustar c]$ in $L$. By left-right symmetry, we can assume that $w$ is to the left of $c$. Listing them from left to right, let $w=w_0, w_1,\dots, w_t$ be the old lower covers of $\ustar c$ that are neither strictly to the left of $w$, nor strictly to the right of $c$; see Figure~\ref{figabwcst} for $t=3$. 
Note that the old elements are empty-filled while the new ones are black-filled, and the elements in the figure do not form a sublattice. Let $w_{t+1}$ be the neighboring old lower cover of $\ustar c$ to the right of $w_t$ in $K$; it is also to the right of $c$. 
By \eqref{eqtxttwcodhBtn}, $\set{w_{i-1}\wedge w_i, w_{i-1}, w_i, \ustar c}$ is a 4-cell of $K$ for $i\in\set{1,\dots,t}$; these 4-cells are colored by alternating shades of grey in the figure. Clearly, $[w_{i-1}, \ustar c]$ strongly swings to $[w_{i}, \ustar c]$ in $K$, for  $i\in\set{1,\dots,t}$. Hence, there is an \sseq{} in $K$ from $[w,\ustar c]=[w_0,\ustar c]$ to $[w_t,\ustar c]$. By \eqref{eqtxtKthnLvn}, we have an \spseq{}  from $[w,\ustar c]$, and thus also from $[a,b]$, to $[w_t,\ustar c]$. 
Also, since $\dstar c$, $\ustar c$, $w_t$, $w_{t+1}$,  and the lower covers of $\ustar c$ between $w_t$ and $w_{t+1}$ form a region in $L$ and a cover-preserving sublattice $M_n$ for some $n$, \eqref{eqtxtM6} allows us to continue the above-mentioned \spseq{}  to $[\dstar c, c]$. Hence, $[a,b]$ \spspan s $[\dstar c, c]=[\dstar c, b\wedge c]$ in $L$.
 This fact and \eqref{eqtxtabspstcc} yield that  $[a,b]\,$ \spspan s $[a\wedge c, b\wedge c]$ in $L$, proving \eqref{eqtxtindPln} for \emph{old edges} $[a,b]$.

Second, we assume that  $[a,b]$ is a new edge. If $b$ is an eye, which has only one lower cover, then $c<b$ gives that $c\leq a$, whence $[a\wedge c,  b\wedge c]$ is a singleton, which is clearly \spspan{}ned. So we can assume that $a$ is an eye with upper and lover covers $\ustar a=b$ and $\dstar a$, respectively. Let $S=\set{\dstar a,w_\ell,w_r,b}$ denote the 4-cell of $K$ into which $a$ has been added. Here this is understood so that several eyes could have been added to this 4-cell simultaneously, whence $[\dstar a,b]_L$ is isomorphic to $M_n$ for some $n\in\set{3,4,\dots}$. 
Applying \eqref{eqtxtM6} to $[\dstar a,b]_L$ and using \eqref{eqtxtKellyRival}, we obtain that
\begin{equation}
\text{$[a,b]$ \spspan{}s both  $[\dstar a, w_r]$ and  $[w_r, b]$ in $L$.}
\label{eqtxtAbsslPsn}
\end{equation}
By the already proved  ``old edge version'' of \eqref{eqtxtindPln}, 
\begin{equation}
\text{$[\dstar a, w_r]$ \spspan{}s $[\dstar a\wedge c, w_r\wedge c]$ and  $[w_r, b]$ \spspan{}s  $[w_r\wedge c, b\wedge c]$.}
\label{eqtxtdzBnWqVy}
\end{equation} 
In \eqref{eqtxtAbsslPsn}, prime intervals are \spspan{}ned, whence  \eqref{eqtxtAbsslPsn} yields \spseq{}s. Combining these \spseq{}s with those provided by \eqref{eqtxtdzBnWqVy} and using transitivity, we obtain that $[a,b]$
\spspan{}s $[\dstar a\wedge c, b\wedge c]$. Hence, we need to show only that $\dstar a\wedge c=a\wedge c$. If we had that $a\leq c$, then $a\prec b$ and $b<c$ would give that $a=c$, contradicting $c\nleq a$. Thus, $a\nleq c$ and $a\wedge c < a$. Since $\dstar a$ is the only lower cover of $a$, we have that $a\wedge c\leq \dstar a$ and so  $a\wedge c\leq \dstar a\wedge c$. Since the converse inequality is obvious,  $a\wedge c = \dstar a\wedge c$, as required. This completes the proof of \eqref{eqtxtindPln}.

\begin{figure}[ht] 
\centerline
{\includegraphics[scale=1.0]{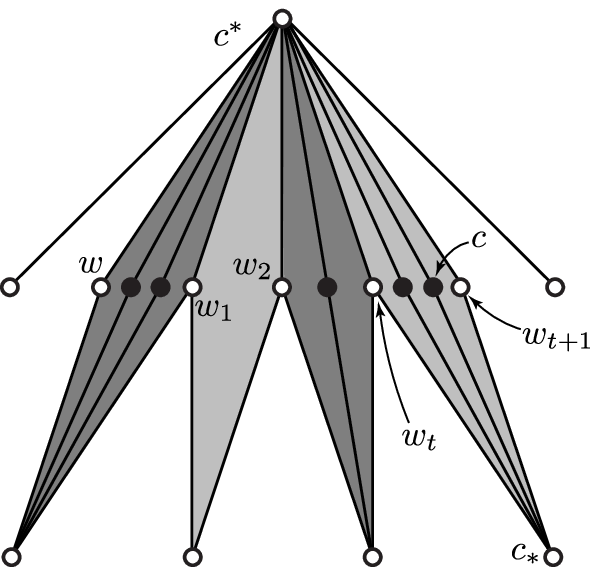}}
\caption{From $[w,\ustar c]$ to $[\dstar c, c]$}\label{figabwcst}
\end{figure}

Next, let $\ba=\{\pair xy\in L^2: \inp\,\text{ \spspan s }[x\wedge y, x\vee y]\}$, where $\inp$ is the prime interval from Theorem~\ref{thmstrongswinglemma}\eqref{thmstrongswinglemmab}.
We are going to show that $\ba$ is a congruence.
Obviously, $\pair x y\in\ba\iff \pair{x\wedge y}{x\vee y}\in\ba$ and 
\begin{equation}
\bigl(x\leq y\leq z, \text{ } \pair x y\in\ba, \text{ and } \pair y z\in\ba \bigr) \Longrightarrow \pair x z\in\ba.
\label{equptrN}
\end{equation} 
Hence, by Gr\"atzer~\cite[Lemma 11]{rGrLTFound}, it suffices to show that whenever $x\leq y$, $\pair x y\in \ba$, and $z\in L$, then $\pair{x\vee z}{y\vee z}\in\ba$ and $\pair{x\wedge z}{y\wedge z}\in\ba$. To do so, pick a maximal chain $x=u_0\prec u_1\prec\dots\prec u_n=y$ that witnesses $\pair x y\in\ba$. Then, for each $i\in\set{1,\dots,n}$, there is an \spseq{} from $\inp$ to $[u_{i-1},u_i]$. By \eqref{eqtxtindPln},  $\pair{u_{i-1}\wedge z}{u_i\wedge z}\in\ba$ for $i\in\set{1,\dots,n}$, and 
\eqref{equptrN} yields that  
$\pair{x\wedge z}{y\wedge z}=\pair{u_{0}\wedge z}{u_n\wedge z}\in\ba$. By semimodularity, 
either  $[{u_{i-1}},{u_i}]$ is up-perspective to  $[{u_{i-1}\vee z},{u_i\vee z}]$, or  ${u_{i-1}\vee z}={u_i\vee z}$. Hence, either by \eqref{eqtxbsLsP} or trivially,  $\pair{u_{i-1}\vee z}{u_i\vee z}\in\ba$. Thus, \eqref{equptrN} implies that  
$\pair{x\vee z}{y\vee z}=\pair{u_{0}\vee z}{u_n\vee z}\in\ba$, and we have shown that $\ba$ is a congruence.

Finally, since $\ba$ collapses $\inp$, we have that $\con(\inp)\subseteq \ba$. So if $\pair{0_\inq}{1_\inq}\in\con(\inp)$, then the containment $\pair{0_\inq}{1_\inq}\in\ba$ and the definition of $\ba$ yield an \spseq{} from $\inp$ to $\inq$. This completes the proof of the slim swing lemma.
\end{proof}

\begin{remark} For a \emph{slim} semimodular lattice $L$, \eqref{eqtxtindPln} is equivalent to \eqref{eqtxtindstp} by \eqref{eqtxthgmnB}. Actually, \eqref{eqtxtindPln} is not needed in this case. In this way, we obtain a proof for the
Swing slim lemma (Lemma~\ref{specswinglemma}) that is much shorter than the proof above.
\end{remark}

\section{Swing lattice game}\label{sectiongame}
In order to describe the essence of our online game, the \emph{Swing lattice game}, we need only two concepts. First, in  Cz\'edli~\cite{czgdiagrectext}, a class $\sf C_2$ of aesthetic slim semimodular lattice diagrams has been introduced. Instead of  repeating the long definition of $\sf C_2$ here, we only mention that the diagrams  in Figures~\ref{figpsa}, \ref{figslima}, and \ref{figindstep} and $L'$ in Figure~\ref{figaddfork} belong to $\sf C_2$, but  
the diagrams in Figure~\ref{figcyclM6} and $L$ in Figure~\ref{figaddfork} do not.
Second,  
an \pseq{} $\vec\inr$ from \eqref{eqsseqpseq} is called an \emph{\gseq} if, for $i\in\set{1,2,\dots,n}$, $\inr_{i-1}\neq \inr_i$. (The acronym comes from ``Swing lemma game''.)  
For the player, who can see the diagram, the exact definition of $\sf C_2$ is not at all important.

In order to avoid  the concept of \gseq{}s, which may cause difficulty for a non-mathematician player, the program 
says simply that a \emph{monkey} keeps moving from edge to edge such that the two edges in question have to belong to the same 4-cell. The monkey can jump or swing or tilt (these steps are easily described in a plain language), but it cannot move back to the edge it came from in the very next step.
The purpose of the game is to make sure that a random \gseq{}  $\vec\inr$ continues as long as possible in a slightly varying diagram $L'$, to be specified later. In the language of the game, which we will use frequently below, the monkey should live as long as the player's luck and, much more significantly, his skill allows. The recent position, $\inr_i$, of the monkey is always indicated by a red thick edge.

At the beginning of the game, the program displays a randomly chosen diagram $L\in\sf C_2$ of a given length. This $L$ is fixed for a while. In order to obtain a bit larger planar semimodular lattice diagram $L'$, the player is allowed to add an eye to one of the 4-cells of $L$ (by a mouse click). Whenever he adds a new eye, the old one disappears; this action is called a \emph{change of the eye}.   In this way, $L'$ is varying but the equality $|L'\setminus L|=1$ always holds. Besides the edges of $L$, which are called \emph{original edges}, $L'$ has two additional edges, the \emph{new edges}. 
In order to  influence the monkey's lifetime, 
\begin{equation}
\text{the player's main tool is to change the eye frequently.}
\label{eqtxtchnEye}
\end{equation}
If the player clicks on a 4-cell while the monkey is moving between two old edges or when it has just arrived at an old edge, then the eye is immediately changed.  However, if the monkey is moving from an old edge to a new one or conversely, then the change is delayed till the monkey arrives at an old edge.
At the beginning,
\begin{equation}
\text{the player has three seconds to choose an edge $\inr_0$ of $L$;}
\label{eqtxtinitedgE}
\end{equation} if he is late, then the computer chooses one randomly. 
After departing from $\inr_0$, the monkey moves at a constant speed at the beginning; later, in order to increase the difficulty, this speed slowly increases.
If the monkey can make several moves, then the program chooses the actual move randomly. From time to time, the program turns a 4-cell into a \emph{bonus cell}, indicated by grey color; if the monkey can jump or swing between two edges of the grey cell within ten moves, then it earns an extra life. Similarly,
the program also offers \emph{candidate cells} in blue color; 
\begin{equation}
\parbox{6.5cm}{if the player accepts the candidate cell by clicking on it within three moves, then this 4-cell becomes a purple \emph{adventure cell}.}
\label{eqtxtadVent}
\end{equation}
The monkey earns two extra lives if it jumps or swings between two edges of the adventure cell within 20 moves but it looses a life otherwise. Also, the monkey looses a life when no move is possible; this can happen only at a boundary edge of the diagram. If a life is lost but the monkey still has at least one life, then the game continues on a new random diagram. When the monkey has no more lives left, the game terminates.

The player, if quick enough, can always save the monkey at boundary edges by using \eqref{eqtxtchnEye}. Also, using \eqref{eqtxtchnEye} appropriately, the player can increase the probability that the monkey will go in a desired direction. In order to make a good decision
how to use \eqref{eqtxtinitedgE}, when to use \eqref{eqtxtadVent}, and when and how to apply \eqref{eqtxtchnEye},  the player should have some experience and insight into the process. Hence, the Swing lattice game is not only a reflex game.

The game is realized by a JavaScript program; see Cz\'edli and Makay~\cite{czgmgthegame}. Most browsers, like Mozilla, can run this program automatically.

The diagrams of length $n$ in $\sf C_2$ are conveniently given by their Jordan-H\"older permutations belonging to the symmetric group $S_n$. 
Since not every diagram in $\sf C_2$ of a given length is appropriate for the game, the program defines the concept of ``good diagrams''. For example, neither a distributive diagram, nor a glued sum decomposable diagram is good.
We have characterized goodness in terms of permutations. Whenever a new diagram is needed,  the program generates a random good permutation $\pi\in S_n$, and the diagram is derived from $\pi$.  
The lattice theoretical background of this algorithm is not quite trivial. However,  instead of going into details in the \emph{present} paper, we only mention that several tools given by Cz\'edli~\cite{czgcoord} and \cite{czgdiagrectext} and Cz\'edli and Schmidt~\cite{czgschperm} have extensively been used.

\color{black}


\begin{thebibliography}{99}


\bibitem{csbjr}
   Cs\'ak\'any, B., Juh\'asz, R.:
   The solitaire army reinspected. 
   Math. Mag. \tbf{73}, 354--362 (2000)

\bibitem{csbhun}
   Cs\'ak\'any, B.: 
   Discrete Mathematical Games. 2nd ed.
   Polygon, Szeged, 2005 (in Hungarian)
    

\bibitem{czgrepres}
    Cz\'edli, G.: 
    Representing homomorphisms of distributive lattices as restrictions of congruences of rectangular lattices.
    Algebra Universalis \tbf{67}, 313--345 (2012)

\bibitem{czgcoord}
    Cz\'edli, G.: 
    Coordinatization of join-distributive lattices.
    Algebra Universalis \tbf{71}, 385--404 (2014)

\bibitem{czgtrajcolor}
    Cz\'edli, G.: Patch extensions and trajectory colorings of slim rectangular lattices.
    Algebra Universalis \tbf{72},  125--154  (2014)

\bibitem{czgdiagrectext}
   Cz\'edli, G.: 
   Diagrams and rectangular extensions of planar
semimodular lattices.
   Algebra Universalis, submitted.  


\bibitem{czgmgthegame}
   Cz\'edli, G., Makay, G.: Swing lattice game program. Available at\\ \texttt{http://www.math.u-szeged.hu/\textasciitilde{}czedli/swinglattice/} \quad or\\
\texttt{http://www.math.u-szeged.hu/\textasciitilde{}makay/swinglattice/}


\bibitem{czgggltsta}
   Cz\'edli, G.,  Gr\"atzer, G.:
   Planar semimodular lattices and their diagrams. Chapter 3 in: Gr\"atzer, G.,
Wehrung, F. (eds.) Lattice Theory: Special Topics and Applications. Birkh\"auser Verlag, Basel (2014)


\bibitem{czgggswing}
   Cz\'edli, G.,  Gr\"atzer, G.:
   Swing Lemma for planar semimodular lattices

\bibitem{czgschthowtoderive}
   Cz\'edli, G., Schmidt, E.T.:
   How to derive finite semimodular lattices from distributive lattices?.  Acta Math. Hungar. \tbf{121}, 277--282 (2008)


\bibitem{czgschtJH}
   Cz\'edli, G., Schmidt, E.T.:
   The Jordan-H\"older theorem with uniqueness for groups and semimodular lattices. 
   Algebra Universalis    \tbf{66}, 69--79 (2011)


\bibitem{czgschvisual}
   Cz\'edli, G., Schmidt, E.T.:
   Slim semimodular lattices. I. A visual approach.
   Order \tbf{29}, 481--497 (2012)

\bibitem{czgschslim2}
   Cz\'edli, G., Schmidt, E.T.:
   Slim semimodular lattices. II. A description by patchwork systems.
   Order \tbf{30}, 689--721  (2013)

\bibitem{czgschperm}
  Cz\'edli, G., Schmidt, E.T.:
  Composition series in groups and the structure of slim semimodular lattices.  
Acta Sci Math. (Szeged) \tbf{79}, 369--390 (2013)

\bibitem{rGrLTFound}
   Gr\"atzer, G.: 
   Lattice Theory: Foundation.
   Birkh\"auser, Basel (2011)

%
\bibitem{swinglemma}
    Gr\"atzer, G.:
    Congruences in slim, planar, semimodular lattices: The Swing Lemma,
    Acta Sci. Math. (Szeged) 81 (2015), 381--397


\bibitem{gratzerknapp1} 
   Gr\"atzer, G., Knapp, E.:
   Notes on planar semimodular lattices. I.  Construction. 
   Acta Sci.\ Math.\ (Szeged) \tbf{73}, 445--462 (2007)
%
%
\bibitem{gratzerknapp3}
   Gr\"atzer, G., Knapp, E.:
   Notes on planar semimodular lattices. III. Congruences of rectangular lattices. 
   Acta Sci. Math. (Szeged), \tbf{75}, 29--48 (2009)


\bibitem{grnation}
   Gr\"atzer, G.,  Nation, J. B.: 
    A new look at the Jordan-H\"older theorem for semimodular lattices.
   Algebra Universalis  \tbf{64}, 309--311 (2010)

\bibitem{kellyrival}
  Kelly, D., Rival, I.: 
  Planar lattices. 
  Canad. J. Math. \tbf{27}, 636--665 (1975)

\end{thebibliography}
\end{document}